\documentclass[12pt,reqno]{amsart}
\usepackage[margin=1in]{geometry}

\usepackage{macros}

\AddToHook{env/prop/begin}{\crefalias{theorem}{prop}}
\AddToHook{env/cor/begin}{\crefalias{theorem}{cor}}
\AddToHook{env/lem/begin}{\crefalias{theorem}{lem}}
\AddToHook{env/que/begin}{\crefalias{theorem}{que}}
\AddToHook{env/conj/begin}{\crefalias{theorem}{conj}}
\AddToHook{env/assu/begin}{\crefalias{theorem}{assu}}
\AddToHook{env/defn/begin}{\crefalias{theorem}{defn}}
\AddToHook{env/nota/begin}{\crefalias{theorem}{nota}}
\AddToHook{env/cond/begin}{\crefalias{theorem}{cond}}
\AddToHook{env/egz/begin}{\crefalias{theorem}{egz}}
\AddToHook{env/prob/begin}{\crefalias{theorem}{prob}}
\AddToHook{env/rmk/begin}{\crefalias{theorem}{rmk}}

\begin{document}

\title[Ray class monoids for orders]{Ray class monoids for orders of number fields}
\author{Gene S. Kopp}
\address{Department of Mathematics, Louisiana State University, Baton Rouge, LA, USA}
\email{kopp@math.lsu.edu}



\date{October 8, 2026}

\begin{abstract} 
We define and study six \emph{ray class monoids} associated to a level datum $(\mathcal{O}; \mathfrak{m}, \Sigma)$, where $\mathcal{O}$ is an order in a number field $K$, $\mathfrak{m}$ is an $\mathcal{O}$-ideal, and $\Sigma$ is a set of real places of $K$. The monoids extend the ray class group (and the ring class group) by including non-invertible classes consisting of non-invertible ideals and ideals not coprime to the modulus. We prove fundamental results on the structure of these ray class monoids.
\end{abstract}

\maketitle

\section{Introduction}\label{sec:introduction}

In abstract class field theory, standard ray class groups---whose elements are ray classes of ideals of the maximal order coprime to a modulus $\mm$---are sufficient describe the Galois groups of a cofinal set of class fields; nonetheless, in many situations, one encounters non-maximal orders, which contain non-invertible ideals, and one encounters ideals not coprime to a relevant modulus or conductor. Such situations arise in explicit class field theory, the arithmetic of abelian varieties, arithmetic statistics, computational number theory, and applications of number theory to cryptography and physics. In those situations, one may still talk about ray classes of ideals, but classes will not generally be invertible. Such classes live in a \emph{monoid} (semigroup with identity) rather than a group.

The present paper introduces and studies \emph{ray class monoids} of orders, which may allow \emph{singular ideal classes} (those whose ideals are non-invertible), \emph{imprimitive ideal classes} (those whose ideals are not coprime to the ray class modulus $\mm$), or both.

\subsection{Taxonomy of ray class monoids}\label{subsec:11}

We describe six different ray class monoids (including the ray class group) obtained by relaxing the invertible ideal condition and the relaxing the coprimality condition in the definition of the ray class group. They are denoted and named as follows.
\begin{center}
\begin{tabular}{|c|c|}
\hline
\textbf{Monoid} & \textbf{Name} \\
\hline
$\Cl_{\mm,\rS}(\OO)= \Clmpa_{\mm,\rS}(\OO)$ & \rcg\ = \rcmpadummy\\
$\Clmpb_{\mm,\rS}(\OO)$ & \rcmpb \\
$\Clmpc_{\mm,\rS}(\OO)$ & \rcmpc \\
$\Clmia_{\mm,\rS}(\OO)$ & \rcmia \\
$\Clmib_{\mm,\rS}(\OO)$ & \rcmib \\
$\Clmic_{\mm,\rS}(\OO)$ & \rcmic \\
\hline
\end{tabular}
\end{center}

These monoids are related by the following inclusions. 
\begin{equation}\label{eq:inclusions}
\begin{tikzcd}
\Clmia_{\mm,\rS}(\OO) \arrow[r,hook] & 
\Clmib_{\mm,\rS}(\OO) \arrow[r,hook] & 
\Clmic_{\mm,\rS}(\OO) \\
\Cl_{\mm,\rS}(\OO) = \Clmpa_{\mm,\rS}(\OO) \arrow[r,hook] \arrow[u,hook] 
& \Clmpb_{\mm,\rS}(\OO) \arrow[r,hook] \arrow[u,hook] 
& \Clmpc_{\mm,\rS}(\OO) \arrow[u,hook]
\end{tikzcd}
\end{equation}
All monoids in this diagram are finite sets, as shown in \Cref{thm:monfin}.

\subsection{Prior work}\label{subsec:12}

The study of semigroups and monoids of ideal classes of orders in number fields was initiated by Dade, Taussky and Zassenhaus \cite{DadeTausskyZassenhaus} in 1962. (They treated Noetherian integral domains, discussed in \Cref{sec:ideals}.) They introduced notions corresponding to the ideal class group and ideal class monoid of an order. 

Previous work of the author and Lagarias \cite{KL-cft} defined and studied ray class groups for non-maximal orders. Consider a \emph{level datum} $(\OO; \mm, \rS)$, where $\OO$ is an order in a number field $K$, $\mm$ is an $\OO$-ideal, and $\rS$ is a subset of $\{\mbox{embeddings } K \inj \R\}$. The associated ray class group is defined to be
\begin{equation}\label{eq:cldef}
\Cl_{\mm,\rS}(\OO) = \frac{\rJ^\ast_\mm(\OO)}{\rP_{\mm,\rS}(\OO)},
\end{equation}
where
\begin{align}
\rJ^\ast_\mm(\OO) &= \{\mbox{invertible fractional $\OO$-ideals coprime to $\mm$}\}; \\
\rP_{\mm,\rS}(\OO) &= \{\alpha\OO : \alpha \equiv 1 \Mod{\mm} \mbox{ and } \rho(\alpha)>0 \mbox{ for all } \rho \in \rS\}.
\end{align}
(A fractional ideal $\cc$ is \emph{coprime to $\mm$} if it can be written as $\cc = \aa\bb^{-1}$ for an $\OO$-ideal $\aa$ and an invertible $\OO$-ideal $\bb$ such that $\aa+\mm = \bb+\mm = \OO$.)

\subsection{Contents of paper}\label{subsec:13}

\Cref{sec:ideals} provides a summary of concepts from commutative algebra used in this paper, with a focus on Noetherian integral domains of dimension $1$ and more specifically orders in number fields. This summary includes a discussion of the fractional ideal properties \emph{invertible}, \emph{ord-invertible}, \emph{divisorial}, and \emph{ord-divisorial}. It also examines the implications of ring properties such as \emph{regular} and \emph{Gorenstein} in the case of orders. Some readers may prefer to skip \Cref{sec:ideals} and return to it as needed.

\Cref{sec:primitive} defines the primitive ray class monoids $\Clmpb_{\mm,\rS}(\OO)$ and $\Clmpc_{\mm,\rS}(\OO)$, obtained by relaxing the invertible ideal condition only, and proves results about their structure. These monoids can be defined as the quotient of a monoid of ideals by (the action of) a certain subgroup of principal ideals. The \rcmpb\ $\Clmpb_{\mm,\rS}(\OO)$ has a particularly nice structure, that of a \emph{Clifford monoid}, characterized by a decomposition as a union of groups indexed by idempotents.

\Cref{sec:imprimitive} defines the imprimitive ray class monoids $\Clmia_{\mm,\rS}(\OO)$, $\Clmib_{\mm,\rS}(\OO)$, and $\Clmic_{\mm,\rS}(\OO)$, obtained by relaxing the coprimality condition (as well as, in the latter two cases, the invertible ideal condition). These monoids require more complicated definitions, in which an ideal monoid is quotiented by an equivalence relation not described by a subgroup action. These monoids are not (generally) Clifford monoids.

\Cref{sec:further} proves that all the ray class monoids under consideration are finite sets.

\Cref{sec:moduli} provides a moduli interpretation of the imprimitive flat ray class monoid. 

\Cref{sec:conclusion} provides some concluding remarks, including a discussion of applications of ray class monoids in current and future work.

\section{Ideal and order properties}\label{sec:ideals}

In this section, we provide a concise overview of desirable properties for ideals of orders, and orders themselves, along with counterexamples. We categorize nonzero fractional ideals of orders based on a hierarchy of ``smoothness'' that will be part of the basis of our hierarchy of ray class monoids. From strongest to weakest, ideals of orders can be invertible (or smooth), ord-invertible (or ord-smooth), or ord-divisorial; they can also be divisorial (implied by ``invertible,'' incomparable with ``ord-invertible'' and ``ord-divisorial''), or they can satisfy none of these properties. We also discuss some properties of orders, from strongest to weakest: maximal, Bass, and Gorestein; orders exist that satisfy none of these. We focus on stating results and giving illustrative examples, outsourcing most proofs as references to the literature.

While the focus is on orders of number fields, we give some results in greater generality, that of Noetherian integral domains of dimension $1$.

\subsection{Types of rings}\label{subsec:21}

In this paper, a \emph{ring} is a commutative ring with unity. The terms \emph{integral domain}, \emph{local ring}, \emph{Noetherian ring}, \emph{discrete valuation ring}, \emph{Krull dimension} (or simply \emph{dimension}), and \emph{integral closure} (of an integral domain in its fraction field) have their standard definitions; see \cite{AM}. 

This paper focuses on orders of number fields; some of its results hold for the larger class of Noetherian integral domains of dimension $1$. A Noetherian integral domain of dimension $1$ is a Noetherian integral domain in which every nonzero prime ideal is maximal.

An \emph{order} in a number field $K$ is subring of $K$ having rank $[K:\Q]$ as a $\Z$-module. All orders $\OO$ of number fields satisfy several special ring properties:
\begin{enumerate}
\item $\OO$ is a Noetherian integral domain of dimension $1$.
\item $\OO$ has finite index as an additive subgroup of its integral closure.
\item The fraction field of $\OO$ has characteristic zero.
\end{enumerate}

The following ring properties do not hold for all orders. They hold in special cases of interest in this paper.
\begin{enumerate}
\item A ring is a \emph{Dedekind domain} if it is a Noetherian integral domain of dimension $1$ that is integrally closed in its quotient field.
\item A local ring is \emph{regular} if its maximal ideal is generated by $n$ elements, where $n$ is the (Krull) dimension of the ring and $n < \infty$.
\item A ring $R$ is \emph{regular} if $R_\mm$ is a regular local ring for every maximal ideal $\mm$ of $R$.
\item A local ring $R$ is \emph{Gorenstein} if $R$ has finite injective dimension as an $R$-module.
\item A ring $R$ is \emph{Gorenstein} if $R_\mm$ is a Gorenstein local ring for every maximal ideal $\mm$ of $R$.
\end{enumerate}

\subsection{Integral and fractional ideals of Noetherian integral domains}\label{subsec:22}

For any Noetherian integral domain $\DD$ with fraction field $K$, we can talk about integral and fractional ideals. An \emph{integral ideal} is simply an ideal, that is, a $\DD$-submodule of $\DD$. A \emph{fractional ideal} is a finitely generated $\DD$-submodule of $K$. Fractional ideals have sum, product, intersection, and quotient (colon) operations; the {\em ideal quotient} is defined as
\begin{equation}
\colonideal{\aa}{\bb} = \{x \in K : x\bb \subseteq \aa\}.
\end{equation}
Some special cases of the ideal quotient are the \emph{multiplier ring} is $\ord(\aa) = \colonideal{\aa}{\aa}$ and, for orders $\OO \subseteq \OO' \subseteq \OO_K$, the \emph{relative conductor} $\ff_{\OO'}(\OO) = \colonideal{\OO'}{\OO}$ and the \emph{(absolute) conductor} $\ff(\OO) = \ff_{\OO_K}(\OO)$. See \cite[Sec.~2--3]{KL-cft} for a detailed treatment of integral and fractional ideals of orders of number fields.

We can define coprimality for pairs of fractional ideals in a Noetherian integral domain.
\begin{defn}\label{defn:coprime}
Let $\DD$ be a Noetherian integral domain and $\cc_1, \cc_2$ be fractional $\DD$-ideals. The ideals $\cc_1$ and $\cc_2$ are \emph{coprime} (or $\cc_1$ is \emph{coprime to} $\cc_2$) if they can be written as $\cc_i = \aa_i\bb_i^{-1}$ with $\aa_i$, $\bb_i$ integral ideals satisfying $\aa_1+\aa_2 = \aa_1+\bb_2 = \bb_1+\aa_2 = \bb_1+\bb_2 = \DD$.
\end{defn}

\subsection{Localization and semi-localization} \label{subsec:23}

We recall the standard definition of localization of a ring (both in general and at a prime ideal) in order to introduce the less common notion of \emph{semilocalization}. A \emph{semilocal ring} is a ring with finitely many maximal ideals.

\begin{defn}
Let $R$ be a ring and $S \subseteq R$. The \emph{localization of $R$ by $S$} is the ring
\begin{equation}
R[S^{-1}]
:= \frac{\left\{\sum_{i=1}^n r_i s_i^\ast : r_i \in R, s_i \in S\right\}}{\langle s s^\ast - 1 : s \in S\rangle}
\end{equation}
obtained by formally adjoining inverses of all the elements of $S$. If $\DD$ is an integral domain, and $\pp$ is a prime ideal of $\DD$, the \emph{localization of $\DD$ at $\pp$} is $\DD_\pp = \DD[(\DD \setminus \pp)^{-1}]$. For any $\DD$-module $M$, we also write $M_\pp := \DD_\pp M$ (the extension of scalars of $M$ by $\DD_\pp$, which is a $\DD_\pp$-module). In particular, the \emph{fraction field} of $\DD$ is $\DD_{\mathbf{0}}$, where $\mathbf{0}$ is the zero ideal. If $\DD$ is Noetherian, and $\aa$ is a fractional $\DD$-ideal such that $\aa \neq \DD$, we define
\begin{equation}
S_\aa = \{s \in \DD : s\DD \mbox{ is coprime to } \aa\},
\end{equation}
and the \emph{semilocalization of $\DD$ away from $\aa$} is $\DD[S_\aa^{-1}]$.
In the special case where $\pp$ is a maximal ideal, note that $\DD_\pp = \DD[S_\pp^{-1}]$.
\end{defn}

\begin{eg}
If $\DD = \Z$ and $\aa = \frac{2}{3}\Z$, then $n\Z$ is coprime to $\aa$ if and only if $2 \ndiv n$ and $3 \ndiv n$. We have
\begin{equation}
\DD[S_\aa^{-1}] = \Z\!\left[\tfrac{1}{p} : p \mbox{ prime, } p \neq 2,3\right] = \left\{\tfrac{a}{b} : a,b \in \Z, \, \gcd(b,6)=1\right\}.
\end{equation}
\end{eg}

\subsection{Primary decomposition}\label{subsec:24}

\begin{defn}
Let $\DD$ be a Noetherian integral domain of dimension $1$ and $\aa$ a nonzero fractional $\DD$-ideal. The \emph{support} of $\aa$ is
\begin{equation}
\supp(\aa) = \supp_\DD(\aa) = \{\pp \mbox{ nonzero prime $\DD$-ideal} : \aa_\pp \neq \DD_\pp\}.
\end{equation}
\end{defn}

While non-maximal orders of number field generally do not have prime factorization of ideals, they do have unique primary decomposition. Here a \emph{primary ideal} $\qq$ is one for which the radical $\rad(\qq)$ is prime, or equivalently, $\supp(\qq)$ is a singleton.

\begin{thm}[primary decomposition]\label{thm:primary}
Let $\DD$ be a Noetherian integral domain of dimension $1$ and $\cc$ a nonzero integral $\DD$-ideal. Then $\cc$ has a unique primary decomposition
\begin{equation}
\cc = \bigcap_{i=1}^n \qq_i = \prod_{i=1}^n \qq_i,
\end{equation}
in which the $\qq_i$ are primary ideals whose radicals $\pp_i = \rad(\qq_i)$ are pairwise distinct.
More generally, if $\cc$ is a nonzero fractional $\DD$-ideal, then there is a unique decomposition
\begin{equation}
\cc = \prod_{i=1}^n \qq_i,
\end{equation}
where $\supp(\qq_i) = \pp_i$ are pairwise distinct and $\supp(\cc) = \{\pp_1, \ldots, \pp_n\}$.
\end{thm}
\begin{proof}
For integral ideals, see \cite[Prop.~9.1]{AM}. The statement for fractional ideals follows by writing $\cc = \aa\bb^{-1}$ for an integral $\DD$-ideal $\aa$ and an invertible integral $\DD$-ideal $\bb$, then computing the primary decompositions of $\aa$ and $\bb$. If $\pp \nin \supp(\cc)$, then $\aa_\pp=\bb_\pp$, so the factors corresponding to $\pp$ in the primary decompositions cancel; otherwise, $\aa_\pp\neq\bb_\pp$, so the corresponding factors do not cancel.
\end{proof}

\subsection{Invertibility of fractional ideals}\label{subsec:25}

We now present equivalent conditions for invertiblility of fractional ideals, and implications between some other ideal properties, for Noetherian integral domains of dimension $1$.

We first define different types of invertibility.

\begin{defn}
Let $\DD$ be a Noetherian integral domain, and $\aa$ be a fractional $\DD$-ideal. 
Let $\DD' = \ord(\aa) = \colonideal{\aa}{\aa}$, which is an order with $\DD \subseteq \DD'$.
\begin{itemize}
\item[(1)] $\aa$ is an \textit{invertible fractional $\DD$-ideal} if and only if there  exists a fractional $\DD$-ideal $\bb$ such that $\aa\bb=\DD$.
\item[(2)] $\aa$ is an \textit{ord-invertible fractional $\DD$-ideal} if and only if $\aa$ is invertible as a fractional $\DD'$-ideal.
\item[(3)] $\aa$ is an \textit{divisorial fractional $\DD$-ideal} if and only if $\colonideal{\DD}{\colonideal{\DD}{\aa}} = \aa$.
\item[(4)] $\aa$ is an \textit{ord-divisorial fractional $\DD$-ideal} if and only if $\colonideal{\DD'}{\colonideal{\DD'}{\aa}} = \aa$.
\end{itemize}
\end{defn}

The following lemma gives criteria for invertibility. 

\begin{lem}\label{lem:invertible}
Let $\DD$ be a Noetherian integral domain of dimension $1$, and let $\aa$ be a nonzero fractional $\DD$-ideal. The following are equivalent.
\begin{enumerate}
\item[(1)] $\aa$ is an invertible fractional $\DD$-ideal (meaning that there exists a fractional $\DD$-ideal $\bb$ such that $\aa\bb=\DD$).
\item[(2)] $\aa \colonideal{\DD}{\aa} = \DD$.
\item[(3)] For every prime ideal $\pp$, the $\DD_\pp$-ideal $\aa_\pp$ is principal.
\end{enumerate}
If $\aa = \pp$ is a prime ideal, then (1--3) are equivalent to
\begin{enumerate}
\item[(4)] $\DD_\pp$ is a regular ring. (Equivalently, $\DD_\pp$ is a discrete valuation ring.)
\end{enumerate}
If in addition $\DD = \OO$ is an order of a number field, then (1--4) are equivalent to
\begin{enumerate}
\item[(5)] $\pp \not\supseteq \ff(\OO)$.
\end{enumerate}
\end{lem}
\begin{proof}
The equivalence of (1) and (2) is a special case of \cite[Lem.~3.6]{KL-cft}. The equivalence of (1) and (3) is \cite[Prop.~3.8]{KL-cft}. 

Now we suppose that $\aa = \pp$ is a prime ideal and show that (3) is equivalent to (4). For primes $\qq \neq \pp$, we have $\pp_\qq = \DD_\qq$, so (3) simply says that $\pp_\pp$ is principal, that is, the maximal ideal of $\DD_\pp$ is principal. This is equivalent to the statement that $\DD_\pp$ is a discrete valuation ring \cite[Prop.~9.2]{AM}, and discrete valuation rings are the same as regular local rings of dimension one \cite[p.~123]{AM}.
For orders, the equivalence of (1) and (5) proven in \cite[Lem.~2.13]{KL-cft}.
\end{proof}

The next lemma gives inclusion relations among the different notions of invertibility.

\begin{lem}\label{lem:divisorial}
Let $\DD$ be a Noetherian integral domain of dimension $1$ having fraction field $K$, and let $\aa$ be a nonzero fractional $\DD$-ideal. Let $\DD' = \ord(\aa) = \colonideal{\aa}{\aa}$.
Consider the following conditions on the ideal $\aa$.
\begin{enumerate}
\item[(1)] $\aa$ is an invertible fractional $\DD$-ideal.
\item[(2)] $\aa$ is an ord-invertible fractional ideal (meaning that $\aa$ is invertible as a fractional $\DD'$-ideal).
\item[(2')] $\aa$ is an invertible fractional $\DD''$-ideal for some ring $\DD'' \subseteq K$.
\item[(3)] $\aa$ is an divisorial fractional $\DD$-ideal (meaning that $\colonideal{\DD}{\colonideal{\DD}{\aa}} = \aa$).
\item[(4)] $\aa$ is an ord-divisorial fractional ideal (meaning that $\colonideal{\DD'}{\colonideal{\DD'}{\aa}} = \aa$).
\item[(5)] $\aa$ is a divisorial fractional $\DD''$-ideal for some ring $\DD'' \subseteq K$ 
\end{enumerate}
Then:
\begin{itemize}
\item (1) $\implies$ (2), (2'), (3), (4), (5).
\item (2) $\iff$ (2').
\item (2) $\implies$ (4), (5) [so (2') $\implies$ (4), (5)].
\item (3) $\implies$ (5).
\item (4) $\implies$ (5).
\item If $\DD$ is an order in a number field, then (5) always holds. 
\item Counterexamples exist in the case where $\DD$ is an order in a number field showing that none of the other possible implications between these statements hold.
\end{itemize}
\end{lem}
\begin{proof}
We first show that (2) $\Leftrightarrow$ (2'). Suppose (2'), that $\aa$ is an invertible fractional $\DD''$-ideal. Then $\DD''\aa = \aa$ (by the definition of ``fractional ideal''), so $\DD'' \subseteq \colonideal{\aa}{\aa} = \DD'$. There exists a $\DD''$-ideal $\bb$ such that $\aa\bb = \DD''$, so $\aa(\bb\DD') = \DD''\DD' = \DD'$, and thus $\bb\DD'$ is the $\DD'$-inverse of $\aa$. The converse direction is immediate on taking $\DD'' = \DD'$.

The implication (1) $\shortimplies$ (2') is trivial (take $\DD'' = \DD$), so we also obtain (1) $\shortimplies$ (2).

To prove the (1) $\shortimplies$ (3), consider an invertible fractional $\DD$-ideal $\aa$. Then, using \cite[Lem.~3.9]{KL-cft},
\begin{equation}
\aa^{-1}\colonideal{\DD}{\colonideal{\DD}{\aa}} = \colonideal{\aa^{-1}}{\colonideal{\DD}{\aa}} = \colonideal{\aa^{-1}}{\aa^{-1}} = \colonideal{\aa^{-1}\aa}{\aa^{-1}\aa} = \colonideal{\DD}{\DD} = \DD,
\end{equation}
so $\colonideal{\DD}{\colonideal{\DD}{\aa}} = \aa$.

The prove that (2) $\shortimplies$ (4) is the same as the prove that (1) $\shortimplies$ (3), replacing $\DD$ by $\DD'$.

The implications (3) $\shortimplies$ (5) and (4) $\shortimplies$ (5) are trivial, and the remaining implications follow logically from those already proven.

If $\DD$ is an order in a number field $K$, choose any $\alpha \in \DD$ such that $K = \Q(\alpha)$. The order $\DD'' = \Z[\alpha]$ is a monogenic suborder of $\DD'$. By \Cref{prop:monogenic}, $\DD''$ is Gorenstein, so by \Cref{prop:gorenstein}, any nonzero fractional $\DD''$-ideal is divisorial. Thus, if $\aa$ is a nonzero fractional $\DD$-ideal, then $\aa$ is also a fractional $\DD''$-ideal, so $\aa$ is $\DD''$-divisorial.

Counterexamples ruling out the other possible implications are given in \Cref{eg:nonimp}.
\end{proof}

\subsection{Dedekind domains, Bass orders and Gorenstein orders}\label{subsec:26}

The next several results classify Noetherian integral domains and orders of number fields in terms of
ring-theoretic properties. 

\begin{prop}\label{prop:regular}
Let $\DD$ be a Noetherian integral domain of dimension $1$. Then the following are equivalent:
\begin{enumerate}
\item[(1)] $\DD$ is a Dedekind domain.
\item[(2)] $\DD$ is   a regular ring.
\item[(3)] $\DD_\pp$ is a discrete valuation ring for every nonzero prime ideal $\pp$.
\item[(4)] Every nonzero integral $\DD$-ideal is invertible.
\item[(5)] Every nonzero fractional $\DD$-ideal is invertible.
\item[(6)] Every primary $\DD$-ideal is a power of a prime ideal. (Thus, primary decomposition gives unique factorization of nonzero ideals into prime ideals.)
\end{enumerate}
If $\DD$ is an order in a number field $K$, then (1)--(6) are equivalent to
\begin{enumerate}
\item[(7)] $\DD = \OO_K$.
\end{enumerate}
\end{prop}
\begin{proof}
The equivalence of (1), (3), and (6) is \cite[Thm.~9.3]{AM}. The equivalence of (1) and (5) is \cite[Thm.~9.8]{AM}. The equivalence of (4) and (5) holds because: Any fractional ideal is of the form $\cc = \aa\bb^{-1}$ for an integral ideal $\aa$ and an invertible integral ideal $\bb$, so $\cc$ is invertible if and only if $\aa$ is invertible. The equivalence of (2) and (3) is shown on \cite[p.~123]{AM}.

For (7), the claim that $\OO_K$ is a Dedekind domain (so (1) holds) is \cite[Thm.~9.5]{AM}. 
Conversely, if $\DD$ is a non-maximal order in a number field $K$, then $\OO_K$ is a fractional $\DD$-ideal, and $\OO_K\bb$ is a fractional $\OO_K$-ideal (as well as a fractional $\DD$-ideal) for any fractional $\DD$-ideal $\bb$, so $\OO_K\bb \neq \OO$. So $\OO_K$ is not invertible as a fractional $\DD$-ideal, and (5) fails. Therefore, for orders of number fields, (7) is equivalent to the other statements.
\end{proof}

\begin{prop}\label{prop:gorenstein}
An order $\OO$ is Gorenstein if and only if every nonzero $\OO$-ideal is divisorial.
\end{prop}
\begin{proof}
This is a special case of \cite[Thm.~6.3]{Bass}.
\end{proof}

\begin{prop}\label{prop:monogenic}
If an order $\OO$ is monogenic (meaning that $\OO = \Z[\alpha]$ for some $\alpha$), then $\OO$ is Gorenstein.
\end{prop}
\begin{proof}
This is a special case of \cite[Prop.~3.6]{JensenThorup}.
\end{proof}

We define a \emph{Bass order} $\OO$ to be one for which every superorder $\OO'$ with $\OO \subseteq \OO' \subseteq \OO_K$ is Gorenstein.

\begin{prop}\label{prop:bass}
Let $\OO$ be an order. The following are equivalent:
\begin{enumerate}
\item[(1)] $\OO$ is Bass.
\item[(2)] Every nonzero $\OO$-ideal is ord-invertible.
\item[(3)] Every $\OO$-ideal is generated by two elements.
\end{enumerate}
\end{prop}
\begin{proof}
See \cite[Prop.~5.2]{JensenThorup}.
\end{proof}

\begin{cor}\label{cor:quadratic}
If $\OO$ is an order in a quadratic number field, then $\OO$ is Bass.
\end{cor}
\begin{proof}
Let $\aa$ be any ideal of $\OO$. Then, as a $\Z$-module,
\begin{equation}
\aa = \alpha_1\Z + \alpha_2\Z
\end{equation} 
for some $\alpha_1, \alpha_2 \in \OO$. In particular, $\aa$ is generated as an ideal by $\{\alpha_1, \alpha_2\}$. So $\OO$ satisfies condition (3) of \Cref{prop:bass}, and thus $\OO$ is Bass.
\end{proof}

\begin{eg}[Order properties; non-implications in \Cref{lem:divisorial}]\label{eg:nonimp}
This example is a continuation of \cite[Ex.~2.5]{KL-cft}.
Let $\alpha = \sqrt[3]{2}$, let $K = \Q(\alpha)$, and consider the orders
\begin{align}
\OO_1 &= \Z[2\alpha] = \Z + 2\alpha\Z + 4\alpha^2\Z; \\
\OO &= \Z[2\alpha,2\alpha^2] = \Z + 2\alpha\Z + 2\alpha^2\Z; \\
\OO' &= \Z[\alpha^2] = \Z + 2\alpha\Z + \alpha^2\Z; \\
\OO_K &= \Z[\alpha] = \Z + \alpha\Z + \alpha^2\Z.
\end{align}
Consider also the $\OO$-ideal
\begin{equation}
\qq = 2\Z + 2\alpha\Z + 4\alpha^2\Z.
\end{equation}
\begin{enumerate}
\item The order $\OO'$ is Bass. It is monogenic, therefore Gorenstein (by \Cref{prop:monogenic}), and the only superorder is $\OO_K$.
\item The order $\OO$ is not Gorenstein. This can be seen by verifying that $\qq$ is not $\OO$-divisorial; indeed,
\begin{equation}
\colonideal{\OO}{\colonideal{\OO}{\qq}} = 2\Z + 2\alpha\Z + 2\alpha^2\Z \neq \qq.
\end{equation}
\item The order $\OO_1$ is Gorenstein but not Bass. It is monogenic, therefore Gorenstein. However, is has a non-Gorenstein superorder, namely $\OO$.
\item The order $\OO'$ itself is (trivially) ord-invertible (thus ord-divisorial), but it is not $\OO$-divisorial. Indeed,
\begin{equation}
\colonideal{\OO}{\colonideal{\OO}{\OO'}} = \Z + \alpha\Z + \alpha^2\Z = \OO_K \neq \OO'.
\end{equation}
\item The ideal $\qq$ is not ord-divisorial and thus is not ord-invertible. Indeed, $\ord(\qq) = \colonideal{\qq}{\qq} = \OO$, and as we have seen in (2), $\qq$ is not $\OO$-divisorial.
\item Nontheless, the ideal $\qq$ is $\OO_1$-divisorial. This must be true because $\OO_1$ is Gorenstein, and it can also be verified by calculation.
\end{enumerate}
\end{eg}

\subsection{Coprimality lemmas} 

Noetherian integral domains, specifically  nonmaximal orders of number fields, do not have unique factorization of ideals and do not have greatest common divisors or least common multiples. They do have a notion  of coprimality for pairs of fractional ideals.

\begin{defn}\label{defn:coprime2}
Let $\DD$ be a Noetherian integral domain and $\cc_1, \cc_2$ be fractional $\DD$-ideals. The ideals $\cc_1$ and $\cc_2$ are \emph{coprime} (or $\cc_1$ is \emph{coprime to} $\cc_2$) if they can be written as $\cc_i = \aa_i\bb_i^{-1}$ with $\aa_i$, $\bb_i$ integral ideals satisfying $\aa_1+\aa_2 = \aa_1+\bb_2 = \bb_1+\aa_2 = \bb_1+\bb_2 = \DD$.
\end{defn}

We prove two useful lemmas about coprimality.
\begin{lem}\label{lem:coprimesupport}
Let $\DD$ be a Noetherian integral domain of dimension $1$, and let $\cc_1$ and $\cc_2$ be fractional $\DD$-ideals. Then $\cc_1$ and $\cc_2$ are coprime if and only if $\supp(\cc_1) \cap \supp(\cc_2) = \emptyset$.
\end{lem}
\begin{proof}
In the case when $\cc_1$ and $\cc_2$ are integral ideals, they are coprime if and only if $\cc_1+\cc_2=\DD$, or equivalently, $(\cc_1)_\pp+(\cc_2)_\pp=\DD_\pp$ for every prime ideal $\pp$. Because $\DD_\pp$ is a local ring, every integral ideal other than $\DD_\pp$ is contained in the maximal ideal, so the latter condition is equivalent to the condition that $(\cc_1)_\pp=\DD_\pp$ or $(\cc_2)_\pp=\DD_\pp$ for every $\pp$, that is, $\supp(\cc_1) \cap \supp(\cc_2) = \emptyset$.

For fractional ideals, we write $\cc_i = \aa_i\bb_i^{-1}$ for integral $\aa_i$ and invertible integral $\bb_i$. Using primary decomposition to cancel common primary factors as in the proof of \Cref{thm:primary} for fractional ideals, we may ensure that $\supp(\aa_i) \cup \supp(\bb_i) = \supp(\cc_i)$. The lemma then follows from the integral ideal case applied to the pairs $(\aa_1,\aa_2)$, $(\aa_1,\bb_2)$, $(\bb_1,\aa_2)$, and $(\bb_1,\bb_2)$.
\end{proof}

\begin{lem}\label{lem:coprimefactors}
Let $\DD$ be a Noetherian integral domain of dimension $1$.
Let $\mm$ be a fractional $\DD$-ideal and $\aa, \bb, \cc$ fractional $\DD$-ideals such that $\aa\bb=\cc$. If any two of $\aa,\bb,\cc$ are coprime to $\mm$, then so is the third.
\end{lem}
\begin{proof}
For every prime $\DD$-ideal $\pp$, we have $\aa_\pp\bb_\pp = \cc_\pp$. If any two of $\aa_\pp$, $\bb_\pp$, and $\cc_\pp$ is equal to $\DD_\pp$, then so is the third. Thus,
\begin{align}
\supp(\aa) \subseteq \supp(\bb) \cup \supp(\cc); &&
\supp(\bb) \subseteq \supp(\aa) \cup \supp(\cc); &&
\supp(\cc) \subseteq \supp(\aa) \cup \supp(\bb).
\end{align}
By \Cref{lem:coprimesupport}, the result follows by considering the intersections with $\supp(\mm)$.
\end{proof}

\section{Primitive ray class monoids}\label{sec:primitive}

We introduce and study monoids containing ray classes of non-invertible ideals coprime to a modulus $\mm$.
  
We first define some (infinite) monoids.
\begin{defn}
If $\mm$ is an ideal of an order $\OO$, set
\begin{align}
\rJ(\OO) &:= \{\mbox{nonzero fractional ideals of } \OO\}, \mbox{ and} \\
\rJ_{\mm}(\OO) &:= \{\mbox{nonzero fractional ideals of } \OO \mbox{ coprime to } \mm\}.
\end{align}
\end{defn}

We now define the \emph{\rcmpc} in terms of these monoids by a quotient operation.

\begin{defn}
For a level datum $(\OO; \mm, \rS)$, the \emph{\rcmpc} is defined to be
\begin{equation}
\Clmpc_{\mm,\rS}(\OO) = \frac{\rJ_{\mm}(\OO)}{\rP_{\mm,\rS}(\OO)}.
\end{equation}
Here $\rP_{\mm,\rS}(\OO)$ is the set of principal ideals  $\alpha \OO$,
with $\alpha \equiv 1 \, (\bmod \mm)$ satisfying the sign conditions in $\rS$ and relatively prime to $\mm$.
We also set $\Clmpc(\OO)=\Clmpc_{\OO,\emptyset}(\OO)$.
\end{defn}

\subsection{Conditions for \rcmpc\ $\Clmpc_{\mm, \rS}(\OO)$ to be a group}\label{subsec:31}

We always have $\Cl_{\mm,\rS}(\OO) = \Clmpa_{\mm, \rS}(\OO) \subseteq  \Clmpc_{\mm,\rS}(\OO) $.
The next result characterizes cases where the sharp primitive ray class monoid coincides with the ray class group.

\begin{prop}\label{prop:clmeqcl}
The following properties are equivalent:
\begin{enumerate}
\item
$\Clmpc_{\mm,\rS}(\OO)$ is a group.
\item
$\Clmpc_{\mm,\rS}(\OO) = \Cl_{\mm,\rS}(\OO)$.
\item
$\rJ_{\mm}(\OO)$ is a group.
\item
$\rJ_{\mm}(\OO) = \rJ_{\mm}^\ast(\OO)$.
\item
$\supp(\mm) \subseteq \supp(\ff(\OO))$.
\item
$\mm \subseteq \rad(\ff(\OO))$.
\end{enumerate}
\end{prop}
\begin{proof}
The implications that (1) $\shortiff$ (2) $\shortiff$ (3) $\shortiff$ (4) are immediate consequences of the definitions together with the observations that principal ideals are invertible and $\rP_{\mm,\rS}(\OO)$ is a group, so an ideal $\aa \in \rJ_\mm(\OO)$ is invertible if and only if its class $[\aa] \in \Clmpa_{\mm,\rS}(\OO)$ is invertible.

We now prove that (4) $\shortiff$ (5). 
If $\supp(\mm) \not\subseteq \supp(\ff(\OO))$, then for any prime ideal $\pp \in \supp(\ff(\OO)) \setminus \supp(\mm)$, it follows from \Cref{lem:invertible}($(1) \shortiff(5)$)  that $\pp \in \rJ_\mm(\OO) \setminus \rJ_\mm^\ast(\OO)$, so $\rJ_\mm(\OO) \neq \rJ_\mm^\ast(\OO)$.
Conversely, if $\supp(\mm) \subseteq \supp(\ff(\OO))$, then by \Cref{lem:invertible}($(1) \shortiff (5)$) every prime ideal in $\rJ_\mm(\OO)$ is invertible, so indeed by the local characterization of invertibility shown in \Cref{lem:invertible}($(1) \shortiff (3)$), every ideal in $\rJ_\mm(\OO)$ is invertible, and $\rJ_\mm(\OO)=\rJ_\mm^\ast(\OO)$.

The implication (5) $\shortiff$ (6) follows from primary decomposition (\Cref{thm:primary}).
\end{proof}

In the special case $\mm = \OO$, it follows from \Cref{prop:clmeqcl} that $\Clmpa(\OO) = \Cl(\OO)$ if and only if $\OO = \OO_K$.

\subsection{The \rcmpb\ $\Clmpb_{\mm,\rS}(\OO)$ }\label{subsec:32}

The structure of $\Clmpc_{\mm,\rS}(\OO)$, or even $\Clmpc(\OO)$, is difficult to describe in general. It may contain elements arising from ideals that are not invertible in any order, and such elements lead to pathological behavior. (See Zanardo and Zannier \cite{ZanardoZannier}.)
However, the submonoid of ord-invertible ideal classes is well-behaved and admits a 
decomposition into groups attached to intermediate orders.

\begin{defn}
Let
\begin{align}
\PI(\OO) &:= \{\mbox{ord-invertible fractional ideals of } \OO\}, \mbox{ and} \\
\PI_{\mm}(\OO) &:= \{\mbox{ord-invertible fractional ideals of } \OO \mbox{ coprime to } \mm\}.
\end{align}
\end{defn}

\begin{defn}
The \emph{\rcmpb} is defined to be
\begin{equation}
\Clmpb_{\mm,\rS}(\OO) = \frac{\PI_{\mm}(\OO)}{\rP_{\mm,\rS}(\OO)}.
\end{equation}
Set $\Clmpb(\OO)=\Clmpb_{\OO,\emptyset}(\OO)$.
\end{defn}

\begin{prop}\label{prop:basseq1}
If $\OO$ is Bass, then $\Clmpb_{\mm,\rS}(\OO) = \Clmpc_{\mm,\rS}(\OO)$. If $\OO$ is not Bass, then $\Clmpb(\OO) \subsetneq \Clmpc(\OO)$.
\end{prop}
\begin{proof}
Follows directly from the equivalence of (1) and (2) in \Cref{prop:bass}.
\end{proof}

To express the niceness of the structure of the \rcmpb, we introduce some terminology. The following definition is paraphrased from \cite{Halter-Koch}; see also \cite{Clifford}.
\begin{defn}\label{defn:clifford}
Let $M$ be a commutative semigroup, and let $\Id(M) = \{e \in M : e^2=e\}$. For each $e \in \Id(M)$, let $M_e = \{x \in M : xe=x \mbox{ and } xy=e \mbox{ for some } y \in M\}$. (It is straightforward to check that each $M_e$ is a group and that the $M_e$ are disjoint.) Then, $M$ is called a \emph{Clifford semigroup} (also called a completely regular inverse semigroup) if
\begin{equation}\label{eq:clifforddecomp}
M = \bigsqcup_{e \in \Id(M)} M_e.
\end{equation}
We call \eqref{eq:clifforddecomp} the \emph{Clifford decomposition} of $M$ into groups.
If in addition $M$ has a (global) identity element, it is called a \emph{Clifford monoid}.
\end{defn}

\begin{prop}\label{prop:picliff}
The \rcmpb\ $\Clmpb_{\mm,\rS}(\OO)$ has the structure of a Clifford monoid. Precisely,
\begin{equation}\label{eq:piu1}
\Clmpb_{\mm,\rS}(\OO) \isom \bigsqcup_{\substack{\OO \subseteq \OO' \subseteq \OO_K \\ \ff_{\OO'}(\OO) + \mm = \OO}} \Cl_{\mm\OO',\rS}(\OO'),
\end{equation}
where multiplication of $[\aa] \in \Cl_{\mm\OO',\rS}(\OO')$ and $[\bb] \in \Cl_{\mm\OO'',\rS}(\OO'')$ is given by
$[\aa][\bb] = [\aa\bb\OO'''] = [\aa\OO'''][\bb\OO''']$ in $\Cl_{\mm\OO''',\rS}(\OO''')$, with $\OO''' = \OO'\OO''$. 
\end{prop}
\begin{proof}
Let $M = \Clmpb_{\mm,\rS}(\OO)$.
Suppose $[\ee] \in \Id(M)$, using the notation from \Cref{defn:clifford}.
Since $\ee$ is ord-invertible, there is some order $\OO'$ with $\OO \subseteq \OO' \subseteq \OO_K$ such that $\ee$ is an invertible $\OO'$-ideal, and there is some fractional $\OO'$-ideal $\dd$ such that $\ee\dd=\OO'$. Thus, $\ee = \ee\OO' = \ee^2\dd = \ee\dd = \OO'$. It follows that $\OO' \in \PI_{\mm}(\OO)$, so $\OO'$ is coprime to $\mm$ over $\OO$. The latter condition means that $\OO' = \aa\bb^{-1}$ for integral $\OO$-ideals $\aa, \bb$ satisfying $\aa+\mm = \bb+\mm = \OO$, so $\bb\OO' = \aa \subseteq \OO$, so $\bb \subseteq \ff_{\OO'}(\OO)$ by the definition of the conductor, and thus $\ff_{\OO'}(\OO) + \mm = \OO$. Therefore,
\begin{equation}
\Id(M) = \{\mbox{orders } \OO' \subseteq K : \OO \subseteq \OO' \subseteq \OO_K\}.
\end{equation}
Moreover, if $\OO'+\mm = \OO$, then
\begin{align}
M_{\OO'} 
&= \{[\aa] \in M : \aa\OO'=\aa \mbox{ and } \aa\bb=\OO' \mbox{ for some } \bb\in M\} \\
&= \{[\aa] \in M : \aa \mbox{ is an invertible $\OO'$-ideal}\} \\
&= \frac{\rJ_{\mm}(\OO) \cap \rJ^\ast(\OO')}{\rP_{\mm,\rS}(\OO)\cdot[\OO']}.
\end{align}
An invertible fractional $\OO'$-ideal is coprime to $\mm$ as an $\OO$-ideal if and only if it is coprime to $\mm\OO'$ as an $\OO'$-ideal, so the numerator $\rJ_{\mm}(\OO) \cap \rJ^\ast(\OO') = \rJ^\ast_{\mm\OO'}(\OO')$. Furthermore, multiplication by the class $[\OO']$ sends two classes of principal ideal $\alpha\OO,\beta\OO \in \rP_{\mm,\rS}(\OO)$ to the same class in $\rJ^\ast_{\mm\OO'}(\OO')$ if and only if $\alpha \equiv_{\OO'} \beta \Mod{\mm\OO'}$ 
and $\alpha,\beta$ have the same signs at the real embeddings in $\rS$. Thus,
\begin{equation}
M_{\OO'}  = \frac{\rJ^\ast_{\mm\OO'}(\OO')}{\rP_{\mm\OO',\rS}(\OO')} = \Cl_{\mm\OO',\rS}(\OO').
\end{equation}
On the other hand, if $\ff_{\OO'}(\OO) + \mm \neq \OO$, then no fractional $\OO'$-ideal is coprime to $\mm$.
By the definition of the \rcmpb, every element $[\cc]$ is in some $M_{\OO'}$, so \eqref{eq:piu1} holds.
\end{proof}

The \textit{order product} $\OO'\OO'' = \left\{\sum_i a_ib_i : a_i \in \OO', b_i \in \OO''\right\}$,
the product of $\OO'$ and $\OO''$ as $\OO$-ideals (for any $\OO \subseteq \OO' \cap \OO''$) used in \Cref{prop:picliff},
is the smallest order containing $\OO'$ and $\OO''$. (This may be seen because $\OO'\OO''$ contains $1$, is closed under ring operations, and is contained in $\OO_K$, so is an order, and any order containing $\OO'$ and $\OO''$ contains all elements of the form $\sum_i a_ib_i$ with $a_i \in \OO'$ and $b_i \in \OO''$.) In many familiar cases, such as quadratic orders, the order product $\OO'\OO''$ is equal to the lattice sum $\OO'+\OO''$. It is clear that $\OO'+\OO'' \subseteq \OO'\OO''$. However, the following example shows that $\OO'+\OO''$ is not always an order and may be strictly smaller than $\OO'\OO''$.

\begin{eg}
Let $K=\Q(2^{1/7})$, and consider the following two orders in $K$.
\begin{align}
\OO' &= \Z[2^{2/7}] = \Z + 2^{2/7}\Z + 2^{4/7}\Z + 2^{6/7}\Z + 2^{8/7}\Z + 2^{10/7}\Z + 2^{12/7}\Z; \\
\OO'' &= \Z[2^{3/7}] = \Z + 2^{9/7}\Z + 2^{18/7}\Z + 2^{6/7}\Z + 2^{15/7}\Z + 2^{3/7}\Z + 2^{12/7}\Z.
\end{align}
Then, the lattice sum of these orders is
\begin{equation}
\OO' + \OO'' = \Z + 2^{2/7}\Z + 2^{4/7}\Z + 2^{6/7}\Z + 2^{8/7}\Z + 2^{3/7}\Z + 2^{12/7}\Z.
\end{equation}
In particular, $2^{2/7}, 2^{3/7} \in \OO' + \OO''$, but their product $2^{5/7} \nin \OO'+\OO''$. So $\OO'+\OO''$ is not an order. The product of orders is equal to
\begin{equation}
\OO'\OO'' = \Z + 2^{2/7}\Z + 2^{4/7}\Z + 2^{6/7}\Z + 2^{8/7}\Z + 2^{3/7}\Z + 2^{5/7}\Z.
\end{equation}
\end{eg}

\subsection{Expressing primitive ray class monoids by equivalence relations}\label{subsec:33}

We will now rewrite each of the three types of primitive ray class monoids as a set of ideals modulo an equivalence relation. The primary purpose of doing so is the introduce the equivalence relations that will be used to define imprimitive ray class monoids in \Cref{sec:imprimitive}. Secondarily, we can deduce that every equivalence class can be represented not only by fractional ideals but also by integral ideals.

\begin{prop}\label{prop:grouptilde2}
The \rcg, \rcmpb, and \rcmpc\ may be written as follows:
\begin{align}
\Cl_{\mm,\rS}(\OO) &\isom \frac{\rJ_\mm^\ast(\OO)}{\sim_{\mm,\rS}} \isom \frac{\rI_\mm^\ast(\OO)}{\sim_{\mm,\rS}}, \\
\Clmpb_{\mm,\rS}(\OO) &\isom \frac{\PI_\mm(\OO)}{\sim_{\mm,\rS}} \isom \frac{\rI_\mm^\natural(\OO)}{\sim_{\mm,\rS}}, \mbox{ and} \\
\Clmpc_{\mm,\rS}(\OO) &\isom \frac{\rJ_\mm(\OO)}{\sim_{\mm,\rS}} \isom \frac{\rI_\mm(\OO)}{\sim_{\mm,\rS}},
\end{align}
where
\begin{equation}
\aa \sim_{\mm,\rS} \bb \iff \begin{array}{c}\exists \cc \in \rJ_{\mm}(\OO) \mbox{ and } \alpha, \beta \in \OO[S_\mm^{-1}] \mbox{ such that } \aa = \alpha\cc, \bb = \beta\cc, \\ \alpha \equiv \beta \Mod{\mm}, \sgn(\rho(\alpha)) = \sgn(\rho(\beta)) \mbox{ for all } \rho \in \rS.\end{array}
\end{equation}
\end{prop}
\begin{proof}
We prove the isomorphisms for $\Clmpc_{\mm,\rS}(\OO)$; the proof the others is the same. Consider $\aa, \bb \in \rJ_\mm(\OO)$. 

If $\aa$ is equivalent to $\bb$ in $\Clmpc_{\mm,\rS}(\OO)$, then there is some $\gamma\OO \in \rP_{\mm, \rS}(\OO)$ such that $\aa = \gamma\bb$, $\gamma \equiv 1 \Mod{\mm}$, and $\rho(\gamma)>0$ for all $\rho \in \rS$. It follows that $\aa \sim_{\mm,\rS} \bb$ by taking $\cc = \bb$ and $(\alpha,\beta) = (1,\gamma)$.

Conversely, suppose $\aa \sim_{\mm,\rS} \bb$. Then there are some $\cc \in \rJ_{\mm}(\OO)$ and $\alpha, \beta \in \OO[S_\mm^{-1}]$ such that $\aa = \alpha\cc$, $\bb = \beta\cc$, $\alpha \equiv \beta \Mod{\mm}$, and $\sgn(\rho(\alpha)) = \sgn(\rho(\beta))$ for all $\rho \in \rS$. Since $\aa, \bb, \cc$ are coprime to $\mm$, it follows (from the equations $\aa=\alpha\cc$ and $\bb=\beta\cc$) that $\alpha\OO$ and $\beta\OO$ are also coprime to $\mm$ (by \Cref{lem:coprimefactors}), so $\alpha\beta^{-1} \equiv 1 \Mod{\mm}$. Also, $\aa = (\alpha\beta^{-1})\bb$, and $\rho(\alpha\beta^{-1})$ for all $\rho \in \rS$, so $\aa$ and $\bb$ are equivalent in $\Clmpa_{\mm,\rS}(\OO)$.

Finally, for the isomorphism $\frac{\rJ_\mm(\OO)}{\sim_{\mm,\rS}} \isom \frac{\rI_\mm(\OO)}{\sim_{\mm,\rS}}$, it suffices to show that every equivalence class in $\frac{\rJ_\mm(\OO)}{\sim_{\mm,\rS}}$ contains an integral ideal. Consider $\dd \in \rJ_{\mm}(\OO)$. Represent $\dd = \aa \bb^{-1}$ for some integral ideals $\aa \in \rI(\OO)$ and some $\bb \in \rI^\ast(\OO)$. There exists some nonzero $\beta \in \bb$ such that $\beta \equiv 1 \Mod{\mm}$ and $\rho(\beta)>0$ for $\rho \in \rS$, so $\beta\dd \in \rI_\mm(\OO)$ and $\dd \sim_{\mm,\rS} \beta\dd$.
\end{proof}

\section{Imprimitive ray class monoids}\label{sec:imprimitive}

Removing the coprimality conditions is more delicate than removing the invertibility conditions. The monoid
$\frac{\rJ^\ast(\OO)}{\rP_{\mm,\rS}(\OO)}$ is infinite for $\mm \neq \OO$ and is not the monoid we want.
We can  remove the coprimality conditions using a characterization with equivalence relations analogous
to those for primitive ray class monoids given in Section \ref{subsec:33}.

\begin{defn}
We define the monoid $\overline{\rJ}^\ast_\mm(\OO)$ (resp.~$\PIt_\mm(\OO)$, $\ol\rJ_\mm(\OO)$) 
of fractional ideals \emph{semilocally integral at $\mm$} that are invertible (resp.~ord-inverible, have no invertibility condition).
\begin{align}
\overline{\rJ}^\ast_\mm(\OO) &= \{\aa \in \rJ^\ast(\OO) : \aa\OO[S_\mm^{-1}] \subseteq \OO[S_\mm^{-1}]\}; \\
\PIt_\mm(\OO) &= \{\aa \in \PI(\OO) : \aa\OO[S_\mm^{-1}] \subseteq \OO[S_\mm^{-1}]\}; \\
\ol\rJ_\mm(\OO) &= \{\aa \in \rJ(\OO) : \aa\OO[S_\mm^{-1}] \subseteq \OO[S_\mm^{-1}]\}.
\end{align}
In all three cases, the semilocal integrality condition $\aa\OO[S_\mm^{-1}] \subseteq \OO[S_\mm^{-1}]$ is equivalent to the condition that $\aa\OO_\pp \subseteq \OO_\pp$ for all nonzero prime ideals $\pp \subseteq \mm$.
\end{defn}

Crucially, ideals that are semilocally integral at $\mm$ need not be coprime to $\mm$; using ideals that are not coprime to $\mm$ will allow us to define a ray class monoid encoding ray class data from both $\mm$ and its proper divisors (more accurately, its proper superideals).

\begin{defn}
Consider the equivalence relation $\sim_{\mm,\rS}$ on $\ol\rJ_\mm(\OO)$ defined by
\begin{equation}
\aa \sim_{\mm,\rS} \bb \iff \begin{array}{c}\exists \cc \in \ol\rJ_{\mm}(\OO) \mbox{ and } \alpha, \beta \in \OO[S_\mm^{-1}] \mbox{ such that } \aa = \alpha\cc, \bb = \beta\cc, \\ \alpha \equiv \beta \Mod{\mm}, \sgn(\rho(\alpha)) = \sgn(\rho(\beta)) \mbox{ for all } \rho \in \rS.\end{array}
\end{equation}
The \emph{\rcmia} is
\begin{equation}
\Clmia_{\mm,\rS}(\OO) = \frac{\overline{\rJ}^\ast_\mm(\OO)}{\sim_{\mm,\rS}}.
\end{equation}
The \emph{\rcmib} is
\begin{equation}
\Clmib_{\mm,\rS}(\OO) = \frac{\PIt_\mm(\OO)}{\sim_{\mm,\rS}}.
\end{equation}
The \emph{\rcmic} is
\begin{equation}
\Clmic_{\mm,\rS}(\OO) = \frac{\ol\rJ_\mm(\OO)}{\sim_{\mm,\rS}}.
\end{equation}
\end{defn}

\subsection{Examples}\label{subsec:40}

The following example calculates the imprimitive ray class monoids in a simple example and illustrates that they are not described simply by quotienting by $\rP_{\mm,\rS}(\OO)$.

\begin{eg}
We examine the simple case where $\OO=\Z$, $\mm=p\Z$ for a prime $p$, and $\rS=\{\infty\}$ consists of the unique infinite place. We have
\begin{equation}
\ol\rJ_{p\Z}^\ast(\Z) 
= \PIt_{p\Z}(\Z)
= \ol\rJ_{p\Z}(\Z)
= \left\{\tfrac{a}{b}\Z : a \neq 0, \, p \ndiv b, \, a,b>0\right\}.
\end{equation}
Note that, for such fractions (with positive numerator and denominator), the equivalence $\tfrac{a_1}{b_1} \sim_{p\Z, \{\infty\}} \tfrac{a_2}{b_2}$ holds if and only if $\tfrac{a_1}{b_1} \equiv \tfrac{a_2}{b_2} \Mod{p}$ in $\Z[S_{(p)}] = \Z_{(p)}$. Thus, the imprimitive ray class monoids are all
\begin{equation}
\Clmia_{\mm,\rS}(\OO) =
\Clmib_{\mm,\rS}(\OO) =
\Clmic_{\mm,\rS}(\OO) =
(\Z/p\Z, \cdot),
\end{equation}
where $(\Z/p\Z, \cdot)$ indicates $\Z/p\Z$ with the multiplication operation.
We also have
\begin{equation}
\rP_{p\Z,\{\infty\}}(\Z) = \left\{\tfrac{a}{b}\Z : p \ndiv a, \, p \ndiv b, \mbox{ and } a \equiv b \Mod{p}\right\}.
\end{equation}
The quotient $\frac{\ol\rJ_{p\Z}^\ast(\Z) }{\rP_{p\Z,\{\infty\}}(\Z)}$ is infinite, because the classes of $\frac{1}{1}\Z$, $\frac{p}{1}\Z$, $\frac{p^2}{1}\Z$, $\frac{p^3}{1}\Z$, and so on are all distinct.
\end{eg}

The next example shows that the  monoid $\Clmia_{\mm,\rS}(\OO)$ is not generally a Clifford monoid (and, therefore, neither are the larger $\Clmib_{\mm,\rS}(\OO)$ and $\Clmic_{\mm,\rS}(\OO)$).

\begin{eg}
If $\mm = \pi^2\OO$ for a principal prime ideal $\pi\OO$, then the only idempotents in $M=\Clmia_{\mm,\rS}(\OO)$ are $[\OO]$ and $[\pi^2\OO]$, and $[\pi\OO] \nin M_e$ for any idempotent $e$. Thus, $M$ is not a Clifford monoid.

A simple case of this phenomenon is the case $\OO=\Z$ and $\mm = p^2\Z$, in which case $M \isom \left(\Z/p^2\Z,\cdot\right)$.
\end{eg}

\subsection{Structure of \rcmia\ $\Clmia_{\mm,\rS}(\OO)$ }\label{subsec:41}

We now give an analogue of \cite[Lem.~5.12]{KL-cft} for the \rcmia, showing that this monoid does not change when additional coprimiality conditions are imposed. 

\begin{lem}\label{lem:45}
For a $K$-order $\OO$, any nonzero ideals $\dd \subseteq \mm \subseteq \OO$, and any set $\rS$ of real places of $K$, the inclusion map $\ol\rJ_\dd^\ast(\OO) \inj \ol\rJ_\mm^\ast(\OO)$ induces a natural monoid isomorphism
\begin{equation}
\Clmia_{\mm,\rS}(\OO) = \frac{\ol\rJ_\dd^\ast(\OO)}{\sim_{\mm,\rS}}
\end{equation}
\end{lem}
\begin{proof}
As in the proof of \cite[Lem.~5.12]{KL-cft}, it suffices to show that the monoid homomorphism
\begin{equation}
\phi : \ol\rJ_\dd^\ast(\OO) \to \frac{\ol\rJ_\mm^\ast(\OO)}{\sim_{\mm,\rS}} = \Clmia_{\mm,\rS}(\OO).
\end{equation}
is surjective. Consider $\bb \in \ol\rJ^\ast(\OO)$. Using primary decomposition, write $\bb = \bb_1\bb_2$, where $\bb_1$ is an integral ideal with the property that $\pp \supseteq \bb_1 \Rightarrow \pp \supseteq \mm$, and $\bb_2$ is a fractional ideal that is coprime to $\mm$. By \cite[Lem.~5.12]{KL-cft}, there exists some $\beta\OO \in \rP_{\mm,\rS}(\OO)$ such that $\beta^{-1}\bb_2 \in \rJ_{\dd}^\ast(\OO)$. It follows that $\beta^{-1}\bb = \bb_1(\beta^{-1}\bb_2) \in \ol\rJ_{\dd}^\ast(\OO)$. We may choose $\beta \equiv 1 \Mod{\mm}$ with $\rho(\beta)>0$ for $\rho \in \rS$; then, the ideal $\aa = \beta^{-1}\bb$ has $\phi(\aa) = \bb$.
\end{proof}

To facilitate describing the structure of the monoid $\Clmia_{\mm,\rS}(\OO)$, we define a notion of an exact sequence of commutative monoids. 
There is no standard definition of such (but see the discussion in \cite{VoightMO}); 
we adopt a definition suited for our application.
As in an exact sequence of abelian groups, we want the fibers of the latter map to be cosets of the image of the former; to guarantee this property, we impose it directly, because it is not sufficient to say that the image of the former map is the kernel of the latter.

\begin{defn}
A sequence
\begin{equation}
\cdots \to A \xrightarrow{\alpha} B \xrightarrow{\beta} C \to \cdots
\end{equation}
of commutative monoids with homomorphisms between them is \emph{exact at $B$} 
if every nonempty preimage of an element $C$ under $\beta$ is a coset of an image of $\alpha$; that is, if for all $c \in C$, either $\beta^{-1}(c) = \emptyset$, or there exists $b \in B$ such that
\begin{equation}
b\alpha(A) = \beta^{-1}(c).
\end{equation}
A sequence that is exact at all objects with an in-arrow and out-arrow is simply called \emph{exact}.
\end{defn}

It should be possible to give analogues of the exact sequence in \cite[Thm.~6.5]{KL-cft} for ray class monoids. For now, we focus on ``resolving'' the map from $\Clmia_{\mm,\rS}(\OO)$ to $\Cl(\OO)$ in order to understand the structure of $\Clmia_{\mm,\rS}(\OO)$.

\begin{prop}\label{prop:exmonoid}
Let $\phi : \Clmia_{\mm,\rS}(\OO) \to \Cl(\OO)$ be the map given by $\phi([\bb]) = [\bb]$. Then, there is an exact sequence of monoids
\begin{equation}
\left(\OO/\mm,\cdot\right) \times \{\pm 1\}^\rS \to \Clmia_{\mm,\rS}(\OO) \xrightarrow{\phi} \Cl(\OO) \to 1.
\end{equation}
\end{prop}
\begin{proof}
Exactness at $\Cl(\OO)$ is equivalent to the surjectivity of $\phi$. By \cite[Lem.~5.12]{KL-cft} (taking $\dd=\mm$), every class in $\Cl(\OO)$ is represented by some $\bb \in \rJ_{\mm}^\ast(\OO)$, and $\rJ_{\mm}^\ast(\OO) \subseteq \overline{\rJ}^\ast_{\mm}(\OO)$, so $\phi$ is surjective.

Define the map $\psi : \left(\OO/\mm,\cdot\right) \times \{\pm 1\}^\rS \to \Clmia_{\mm,\rS}(\OO)$ by
\begin{equation}
\psi(\ol{\alpha},\epsilon) = [\alpha\OO] \mbox{ where } \alpha \equiv \ol{\alpha} \Mod{\mm} \mbox{ and } \sgn(\rho(\alpha)) = \epsilon_\rho.
\end{equation}
This map is well-defined because:
\begin{itemize}
\item[(i)] For any pair $(\ol{\alpha},\epsilon)$, the set $(\ol{\alpha} + \mm) \cap \{\alpha \in \OO : \sgn(\rho(\alpha)) = \epsilon_\rho\} \neq \emptyset$.
\item[(ii)] If $\alpha_1, \alpha_2 \in \OO \setminus \{0\}$, $\alpha_1 \equiv \alpha_2 \Mod{\mm}$, and $\sgn(\rho(\alpha_1))=\sgn(\rho(\alpha_2))$, then $[\alpha_1\OO] = [\alpha_2\OO]$ in $\Clmia_{\mm,\rS}(\OO)$.
\end{itemize}

To prove exactness at $\Clmia_{\mm,\rS}(\OO)$, consider a class $\BB \in \Cl(\OO)$. By \cite[Lem.~5.12]{KL-cft}, 
we may write $\BB = [\bb]$ for some $\bb \in \rJ^\ast_\mm(\OO)$.
Clearly $\phi([\alpha\bb]) = \BB$ for any $\alpha \in \OO[S_\mm^{-1}] \setminus \{0\}$, 
so $\{[\alpha\bb] : \alpha \in \OO[S_\mm^{-1}] \setminus \{0\}\} \subseteq \phi^{-1}(\BB)$.
On the other hand, suppose $\aa \in \overline{\rJ}^\ast_{\mm}(\OO)$ such that $\phi([\aa]) = \BB$. 
Then $\aa$ is equivalent to $\bb$ in $\Cl(\OO)$, so $\aa = \alpha\bb$ for some $\alpha \in K^\times$. 
Moreover, $\alpha\OO = \aa\bb^{-1} \in \overline{\rJ}^\ast_{\mm}(\OO)$, that is, $\alpha\OO$ 
is semilocally integral at $\mm$, so $\alpha \in \OO[S_\mm^{-1}] \setminus \{0\}$. 
Therefore, $\{[\alpha\bb] : \alpha \in \OO[S_\mm^{-1}] \setminus \{0\}\} = \phi^{-1}(\BB)$. 
The left-hand set is the same as $[\bb]\phi\!\left(\left(\OO/\mm,\cdot\right) \times \{\pm 1\}^\rS\right)$, 
so we have proven that the sequence is exact at $\Clmia_{\mm,\rS}(\OO)$.
\end{proof}

The exact sequence in \Cref{prop:exmonoid} is related to the ray class group by the following commutative diagram,
where in both rows the image of the first map consists of the classes of principal ideals.
\begin{equation}
\begin{tikzcd}[column sep=small] 
& \left(\OO/\mm\right)^\times \times \{\pm 1\}^\rS \arrow[r] \arrow[d, hook] 
& \Cl_{\mm,\rS}(\OO) \arrow[r] \arrow[d, hook] 
& \Cl(\OO) \arrow[r] \arrow[d, equal] 
& 1
\\ 
& \left(\OO/\mm,\cdot\right) \times \{\pm 1\}^\rS \arrow[r] 
& \Clmia_{\mm,\rS}(\OO) \arrow[r]
& \Cl(\OO) \arrow[r]
& 1
\\
\end{tikzcd}
\end{equation}

\subsection{The \rcmib\ $\Clmib_{\mm,\rS}(\OO)$ and the \rcmic\ $\Clmic_{\mm,\rS}(\OO)$}\label{subsec:42}

There is no analogue of \cite[Lem.~5.12]{KL-cft} for the non-flat ray class monoids because introducing coprimality conditions to the conductor will eliminate singular classes. We do have the following analogue of \Cref{prop:basseq1}.

\begin{prop}\label{prop:basseq2}
If $\OO$ is Bass, then $\Clmib_{\mm,\rS}(\OO) = \Clmic_{\mm,\rS}(\OO)$. If $\OO$ is not Bass, then $\Clmib(\OO) \subsetneq \Clmic(\OO)$.
\end{prop}
\begin{proof}
Follows directly from the equivalence of (1) and (2) in \Cref{prop:bass}.
\end{proof}

\subsection{Submonoids of zero classes}

An imprimitive ray class monoid contains a submonoid the we call the \emph{submonoid of zero classes}, consisting of those classes that are the ``most imprimitive.''

\begin{defn}
The \emph{flat submonoid of zero classes}, \emph{natural submonoid of zero classes}, and \emph{sharp submonoid of zero classes} are defined as follows, respectively.
\begin{align}
\ZClmia_{\mm,\rS}(\OO) &= \left\{[\dd] \in \Clmia_{\mm,\rS}(\OO) : \dd \subseteq \mm \right\}; \\
\ZClmib_{\mm,\rS}(\OO) &= \left\{[\dd] \in \Clmib_{\mm,\rS}(\OO) : \dd \subseteq \mm \right\}; \\
\ZClmic_{\mm,\rS}(\OO) &= \left\{[\dd] \in \Clmic_{\mm,\rS}(\OO) : \dd \subseteq \mm \right\}.
\end{align}
\end{defn} 

Zero classes must be treated separately in certain applications. For example, they are explicitly excluded from cohomological second Kronecker limit formula for real quadratic fields given as \cite[Thm.~1.1]{kopp}; they would instead need to be treated by an analogue of the first Kronecker limit formula.

\section{Finiteness of ray class monoids}\label{sec:further}

In this section, we prove that all ray class monoids in our taxonomy are finite. 
We begin with a proposition showing that every sufficiently large power of every nonzero fractional ideal of an order is ord-inverible, with a uniform bound on the power.

\begin{prop}\label{prop:invertible} 
Let $\OO$ be an order of a number field $K$.
Then, for any nonzero fractional $\OO$-ideal $\aa$, there is a positive power $N_{\aa} \ge 1$ such that 
the ideal $\aa^{N}$ is an ord-invertible ideal if and only if $N \ge N_{\aa}$.
Furthermore $N_{\aa}$ is bounded uniformly by $N_{\aa} \leq n-1$, where $n=[K: \Q]$.
\end{prop} 
\begin{proof}
The proposition is a special case of \cite[Thm.~C, Cor.~2.2.17]{DadeTausskyZassenhaus}, taking $\OO_E =\Z$ , $E=\Q$, and $\OO$ an order of $F=K$. (Dade, Taussky, and Zassenhaus define ``invertibility'' of a fractional ideal $\aa$ to mean ord-invertibility, that is invertibility with respect to the order $\ord(\aa)=\colonideal{\aa}{\aa}$. Note that it is possible that $\ord(\aa)$ is a strict subset of $\ord(\aa^N)$.) 
\end{proof} 

We prove the general finiteness result using the finiteness of the class group together with \Cref{prop:invertible}.
\begin{thm}\label{thm:monfin}
Let $\OO$ be an order in a number field $K$, $\mm$ an ideal of $\OO$, and $\rS$ a subset of the set of real embeddings of $K$.
The monoids 
$\Cl_{\mm,\rS}(\OO)$,
$\Clmpb_{\mm,\rS}(\OO)$,
$\Clmpc_{\mm,\rS}(\OO)$,
$\Clmia_{\mm,\rS}(\OO)$,
$\Clmib_{\mm,\rS}(\OO)$, and
$\Clmic_{\mm,\rS}(\OO)$
are finite. 
\end{thm}
\begin{proof}
It suffices to the that $\Clmic_{\mm,\rS}(\OO)$ is finite because each of the other monoids injects into this one. Consider the monoid homomorphism $\Clmic_{\mm,\rS}(\OO) \to \Clmpa(\OO)$ given by sending the class $[\aa]$ to the class $[\aa]$. Two classes $[\aa],[\bb] \in \Clmic_{\mm,\rS}(\OO)$ get sent to the same place if and only if $\bb = \gamma \aa$ for some $\gamma \in K^\times$. If $\gamma_1, \ldots, \gamma_k \in \OO[S_\mm^{-1}]$ are a complete set of representatives of the fibers of the monoid homomorphism
\begin{align}
\left(\OO[S_\mm^{-1}],\cdot\right) &\to \left(\OO/\mm,\cdot\right) \times \{\pm 1\}^\rS, \\
\gamma &\mapsto \left(\gamma \Mod{\mm}, (\rho(\gamma) : \rho \in \rS)\right),
\end{align}
then the classes in $\Clmic_{\mm,\rS}(\OO)$ sent to the same place as $[\aa]$ are $[\gamma_1\aa], \ldots, [\gamma_k\aa]$.

It thus suffices to show $\Clmpa(\OO)$ is finite. The extension map $\ext: \rJ(\OO) \to \rJ(\OO_K)$ given by $\ext(\aa) = \aa\OO_K$ induces a monoid homomorphism
\begin{equation}
\ol\ext : \Clmpa(\OO) \to \Clmpa(\OO_K) = \Cl(\OO_K);
\end{equation}
we need to show that $\ol\ext$ is finite-to-one. Consider $[\bb] \in \Cl(\OO_K)$. If $\ol\ext([\aa_1]) = [\bb]$, then $\aa_1\OO_K = \alpha\bb$ for some $\alpha \in K$, so for $\aa := \alpha^{-1}\aa_1$ we have $[\aa]=[\aa_1]$, $\ol\ext([\aa]) = [\bb]$, and $\aa\OO_K = \bb$. Let $m = [K : \Q] - 1$. By \Cref{prop:invertible}, $\aa^m$ is $\OO'$-invertible for some order $\OO'$ with $\OO \subseteq \OO' \subseteq \OO_K$. Using the order ideal norm defined by $\Nm_\OO(\aa) := [\OO : \aa]$, we have by \cite[Prop.~A.8]{KL-cft} that $\Nm_{\OO_K}\!(\bb^m) = \Nm_{\OO_K}\!(\aa^m\OO_K) = \Nm_{\OO'}\!(\aa^m)$. Thus,
\begin{equation}
[\OO : \aa] \leq [\OO : \aa^m] \leq [\OO' : \aa^m] = \Nm_{\OO'}\!(\aa^m) = \Nm_{\OO_K}\!(\bb^m) = \Nm_{\OO_K}\!(\bb)^m.
\end{equation}
There are only finitely many sublattices of $\OO$ having index bounded by $\Nm_{\OO_K}\!(\bb)^m$, so there are only finitely many possibilities for $\aa$, thus for $[\aa]$. We have thus shown that $\ol\ext$ is finite-to-one, so it follows by the finiteness of $\Cl(\OO_K)$ that $\Clmpa(\OO)$ is finite, and so is $\Clmpa_{\mm,\rS}(\OO)$.
\end{proof}

\section{Moduli interpretation of ray class monoids}\label{sec:moduli}

In this section, we give an interpretation of classes in flat imprimitive ray class monoids as special $\GL_n(\Z)$-orbits in a certain moduli space. The results of this section generalize \cite[Thm.~3.14]{kopp}, which deals with the real quadratic case. These results also provide a motivation for imprimitive ray class monoids, in that some $\GL_n(\Z)$-orbits are only accessible via imprimtive classes.

\begin{thm}\label{thm:corresp}
Let $\OO$ be an order in a number field $K$ with $n=[K:\Q]$. Let $\mm$ be an invertible $\OO$-ideal, and let $\Sigma$ be a subset of the real embeddings of $K$. There is an explicit injective function
\begin{equation}
\Clmia_{\mm,\Sigma}(\OO) \xrightarrow{\Upsilon_{\mm,\Sigma}} \GL_n(\Z) \backslash \left(\{\pm 1\}^\Sigma \times (\Q/\Z)^n \times K^n\right) \slash K^\times
\end{equation}
where the action of $\GL_n(\Z)$ is 
\begin{equation}
A\cdot(\vs,\vr,\vb) = (\vs, A^{-\top}\vr, A\vb)
\end{equation}
and the action of $K^\times$ is
\begin{equation}
(\vs,\vr,\vb)\cdot\lambda = (\sgn_\Sigma(\lambda)\vs, \vr, \lambda\vb).
\end{equation}
The image of $\Upsilon_{\mm,\Sigma}$ is described as
\begin{align}
\image(\Upsilon_{\mm,\Sigma}) &= \GL_n(\Z) \backslash \M_{\mm,\Sigma} \slash K^\times \\
\M_{\mm,\Sigma} &= \{(\vs,\vr,\vb) : (\vb\cdot\vr)\mm \subseteq \vb\cdot\Z^n \text{ and } \vb\cdot\Z^n \text{ is an $\OO$-invertible ideal}\}.
\end{align}
The zero class are $\ZClmia_{\mm,\rS}(\OO) = \Upsilon_{\mm,\rS}^{-1}(\{\pm 1\}^\rS \times \{\zero\} \times K^n)$.

The function $\Upsilon_{\mm,\Sigma}$ is defined explicitly as follows for a ray ideal class $\AA \in \Clmia_{\mm,\Sigma}(\OO)$. Let $\AA_0$ be the class of $\AA$ in the narrow ring class group $\Cl_{\OO,\rS}(\OO)$. Choose any $\bb \in \AA_0^{-1}$ and write $\bb\mm = \vb\cdot\Z^n$. Choose any $\gamma \in \bb$ such that $\gamma\OO \in \bb\AA$ and write $\gamma=\vb\cdot\vr$ for some $\vr\in\Q^n$. Set
\begin{equation}
\Upsilon_{\mm,\Sigma}(\AA) = (\sgn_\Sigma(\gamma),\vr,\vb).
\end{equation}
\end{thm}
\begin{proof}
We will first show that the map $\Upsilon_{\mm,\Sigma}$ is well-defined, that is, it does not depend on the choices of $\bb$, $\vb$, $\gamma$. (Note the $\vr$ is determined by $\vb$ and $\gamma$.)
Consider two such tuples of choices $(\bb_1,\vb_1,\gamma_1)$ and $(\bb_2,\vb_2,\gamma_2)$, as well as corresponding $\vr_1$ and $\vr_2$ such that $\gamma_j = \vb_j \cdot \vr_j$. We have $\bb_j \mm = \vb_j \cdot \Z^n$. Moreover, $\bb_1\mm$ and $\bb_2\mm$ are both in the ideal class $\mm\AA_0^{-1}$, so there exists some $\delta \in \OO[S_\mm^{-1}]^\times$ such that $\bb_1\mm = \delta\bb_2\mm$, that is, $\vb_1 \cdot \Z^n = \delta\vb_2 \cdot \Z^n$, and $\rho(\delta) > 0$ for $\delta \in \rS$. Thus, there exists a change-of-basis matrix $A \in \GL_n(\Z)$ such that
\begin{equation}\label{eq:corresp1}
\vb_1 = A(\delta \vb_2) = \delta A\vb_2.
\end{equation}
By assumption, $\mm$ is an invertible $\OO$-ideal, so $\bb_1 = \delta \bb_2$, that is, $\delta^{-1}\bb_1 = \bb_2$. Also, $\delta^{-1}\gamma_1\OO \in \delta^{-1}\bb_1\AA = \bb_2\AA$ and $\gamma_2\OO \in \bb_2\AA$, so it follows from the definition of $\Clmia_{\mm,\rS}(\OO)$ that there exists a global unit $\e \in \OO^\times$ such that $\e\delta^{-1}\gamma_1 \equiv \gamma_2 \Mod{\mm}$ and $\sgn_\rS(\e\delta^{-1}\gamma_1) = \sgn_\rS(\gamma_2)$ (where equivalence $\!\Mod{\mm}$ will always be taken in the semilocal ring $\OO[S_\mm^{-1}]$). We may write
\begin{equation}\label{eq:corresp2}
\e\vb_2 = E\vb_2
\end{equation}
for some $E \in \GL_n(\Z)$. It then follows from \eqref{eq:corresp1} that
\begin{equation}
\e\delta^{-1}\vb_1 = AE\vb_2.
\end{equation}
Moreover, $\gamma_2 \in \bb_2$ and $\e\delta^{-1}\gamma_1 \in \delta^{-1}\bb_1 = \bb_2$, so
\begin{equation}
\e\delta^{-1}\gamma_1 - \gamma_2 \in \bb_2 \cap \mm\OO[S_\mm^{-1}] = \bb_2\mm.
\end{equation}
Additionally,
$\e\delta^{-1}\gamma_1 = \e\delta^{-1} \vb_1\cdot\vr_1 = \e(A\vb_2)\cdot\vr_1 = (AE\vb_2)\cdot\vr_1 = \vb_2\cdot((AE)^\top\vr_1)$ and $\gamma_2 = \vb_2\cdot\vr_2$, so
\begin{equation}
\vb_2\cdot((AE)^\top\vr_1 - \vr_2) \in \bb_2\mm = \vb_2\cdot\Z^n.
\end{equation}
Thus, $(AE)^\top\vr_1 - \vr_2 \in \Z^n$, so
\begin{equation}
\vr_1 \equiv (AE)^{-\top}\vr_2 \Mod{\Z^n}.
\end{equation}
Therefore,
\begin{align}
AE \cdot (\sgn_\rS(\gamma_2), \vr_2, \vb_2)
&= (\sgn_\rS(\gamma_2), (AE)^{-\top}\vr_2, AE\vb_2) \\
&= (\sgn_\rS(\gamma_2), \vr_1, \e\delta^{-1}\vb_1) \\
&= (\sgn_\rS(\e\delta^{-1}\gamma_2), \vr_1, \vb_1) \cdot \e\delta^{-1} \\
&= (\sgn_\rS(\gamma_1), \vr_1, \vb_1) \cdot \e\delta^{-1},
\end{align}
so the map $\Upsilon_{\mm,\rS}$ is well-defined.

To prove that $\Upsilon_{\mm,\rS}$ is injective and has the image specified in the theorem statement, we will construct a function
\begin{equation}
\Omega_{\mm,\rS} : \GL_n(\Z) \backslash \M_{\mm,\rS} / K^\times \to \Clmia_{\mm,\rS}(\OO)
\end{equation}
and show that $\Omega_{\mm,\rS}$ defines an inverse to $\Upsilon_{\mm,\rS}$.
Consider $(\vs,\vr+\Z^n,\vb) \in \M_{\mm,\rS}$. 
Represent $\vr+\Z^n \in (\Q/\Z)^n$ by a choice of lift $\vr \in \Q^n$ such that $\sgn_\rS(\vb\cdot\vr)=\vs$.
Set $\bb = \colonideal{\vb\cdot\Z^n}{\mm}$; it follows that $\vb\cdot\Z^n = \bb\mm$ and $\bb$ is an invertible $\OO$-ideal. Also set $\gamma = \vb\cdot\vr$.
Define
\begin{equation}
\Omega_{\mm,\rS}(\vs,\vr,\vb)
= [\gamma\bb^{-1}] \in \Clmia_{\mm,\rS}(\OO).
\end{equation}

To show that $\Omega_{\mm,\rS}$ is well-defined, consider $(\vs_1,\vr_1+\Z^n,\vb_1)$ and $(\vs_2,\vr_2+\Z^n,\vb_2)$ in $\M_{\OO,\mm}$, with the $\vr_j$ lifted to $\Q^n$ such that $\sgn_\rS(\vb_j\cdot\vr_j)=\vs_j$, for which
\begin{equation}
(\vs_2,\vr_2,\vb_2) 
= A \cdot (\vs_1,\vr_1+\vn,\vb_1) \cdot \lambda
= (\sgn_\rS(\lambda)\vs_1,A^{-\top}(\vr_1+\vn),\lambda A\vb_1) \label{eq:corresp8}
\end{equation}
for some $\vn \in \Z^n$, some $A \in \GL_n(\Z)$, and some $\lambda \in K^\times$. These determine $\bb_j = \colonideal{\vb_j\cdot\Z^n}{\mm}$ and $\gamma_j = \vb_j \cdot \vr_j$. We have
\begin{align}
\gamma_2
= \vb_2 \cdot \vr_2
&= (\lambda A\vb_1) \cdot (A^{-\top}(\vr_1+\vn)) \\
&= \lambda(\vb_1\cdot\vr_1 + \vb_1\cdot\vn) \\
&= \lambda(\gamma_1 + \vb_1\cdot\vn).
\end{align}
Because $A \in \GL_n(\Z)$, we also have
\begin{align}
\bb_2 
= \colonideal{\vb_2\cdot\Z^n}{\mm}
= \lambda \colonideal{A\vb_1\cdot\Z^n}{\mm}
= \lambda \colonideal{\vb_1\cdot\Z^n}{\mm}
= \lambda \bb_1. \label{eq:corresp10}
\end{align}
It follows that $\gamma_2\bb_2^{-1} = (\gamma_1+\vb_1\cdot\vn)\bb_1^{-1}$. We have $\vb_1\cdot\vn \in \vb_1\cdot\Z^n = \bb_1\mm$, and $\bb_1\mm\OO[S_\mm^{-1}] = \mm\OO[S_\mm^{-1}]$. Therefore, $\gamma_1+\vb_1\cdot\vn \equiv \gamma_1 \Mod{\mm}$, and thus
\begin{equation} \label{eq:corresp11}
\gamma_2\bb_2^{-1} \sim_{\mm} \gamma_1\bb_1^{-1}.
\end{equation}
It remains to check that the signs of the real embeddings agree. By \eqref{eq:corresp10}, $\vs_2 = \sgn_\rS(\gamma_2) = \sgn_\rS(\lambda) \sgn_\rS(\gamma_1+\vb_1\cdot\vn)$; also, $\vs_1 = \sgn_\rS(\gamma_1)$. But $\vs_2 = \sgn_\rS(\lambda)\vs_1$ by \eqref{eq:corresp8}, and thus $\sgn_\rS(\gamma_1+\vb_1\cdot\vn) = \sgn_\rS(\gamma_1)$. Therefore \eqref{eq:corresp11} can be upgraded to
\begin{equation} \label{eq:corresp12}
\gamma_2\bb_2^{-1} \sim_{\mm,\rS} \gamma_1\bb_1^{-1},
\end{equation}
and so the map $\Omega_{\mm,\rS}$ is well-defined.

It remains to show that $\Omega_{\mm,\rS} \circ \Upsilon_{\mm,\rS}$ is the identity on $\Clmia_{\mm,\rS}(\OO)$ and $\Upsilon_{\mm,\rS} \circ \Omega_{\mm,\rS}$ is the identity on $\GL_n(\Z) \backslash M_{\mm,\rS} / K^\times$. But these claims following directly from an algebraic check using the definitions of the two functions.
A direct calculation also shows that $\ZClmia_{\mm,\rS}(\OO) = \Omega_{\mm,\rS}\!\left(\left\{(\vs,\vr,vb)\in M_{\mm,\rS} : \vr \equiv \zero \Mod{\Z^n}\right\}\right) = \Upsilon_{\mm,\rS}^{-1}\!\left(\{\pm 1\}^\rS \times \{\zero\} \times K^n\right)$.
\end{proof}

\section{Concluding remarks}\label{sec:conclusion}

We have developed a taxonomy of ray class monoids and proven various results about their structure. We introduced two monoids of primitive (but potentially singular) ideal classes, the \rcmpb\ $\Clmpb_{\mm,\rS}(\OO)$ and the \rcmpc\ $\Clmpc_{\mm,\rS}(\OO)$; the \rcmpb\ has the particularly nice structure of a Clifford monoid. We also introduced three monoids of imprimitive ideal classes, the \rcmia\ $\Clmia_{\mm,\rS}(\OO)$, the \rcmib\ $\Clmib_{\mm,\rS}(\OO)$, and the \rcmic\ $\Clmic_{\mm,\rS}(\OO)$. All were proved to be finite.

These monoids are already being applied to explicit class field theory for real quadratic fields and to the problem of classification of \emph{SIC-POVMs} (symmetric, informationally complete positive operator-valued measures). The \rcmia\ acts as a parameter space for \emph{real multiplication values} of the \emph{Shintani--Faddeev cocycle} \cite{kopp}, which are a refinement in the real quadratic case of the class invariants called \emph{Stark units}. \emph{SIC-POVMs}, or generalized quantum measurements described by configurations of $d^2$ equiangular complex lines in $\C^d$, are conjectured \cite{KL-sic,afk} to be in bijection with the elements of $\Clmpa(\OO)$ for $\OO = \Z\!\left[\frac{d-1+\sqrt{(d+1)(d-3)}}{2}\right]$.

More generally, ray class monoids are intended to be a widely applicable tool in explicit class field theory, beyond the quadratic case.
They are also likely to find applications in the theory of higher-dimensional abelian varieties with complex multiplication, where the Gorenstein condition plays an important role in ensuring nice behavior \cite{Clark}.
Nonetheless, there are many subtleties to the behavior of non-invertible ideal classes and objects (including zeta functions) associated to them, requiring further study.

\section{Acknowledgements}

Work of the author was supported by NSF DMS grants \#2302514 and \#2601235. 

This paper was originally planned as an appendix to \cite{KL-cft} with coauthor Jeffrey C.~Lagarias. While the ideas and writing of the present paper are due to the author, many conversations with Lagarias played in important role in the development of the paper. In particular, Lagarias found some mathematical mistakes in an earlier draft, located \Cref{prop:invertible} and references \cite{Clifford,DadeTausskyZassenhaus,Halter-Koch} in the literature, made the author aware of the concepts of \emph{Clifford semigroup} and \emph{Clifford monoid}, and offered feedback on the structure and writing of the paper. The author is grateful to Lagarias for his contributions.

The author is also grateful to Pete L.~Clark for helpful conversations about Gorenstein and non-Gorenstein orders and their role in the study of abelian varieties.

This paper was written without AI assistance.

\bibliographystyle{abbrv}
\bibliography{references}

\end{document}